\documentclass{article}
\usepackage{graphicx} 
\usepackage{amssymb}
\usepackage{amsmath}
\usepackage{amsfonts}
\usepackage{amsthm}
\usepackage{stmaryrd}
\usepackage{xcolor}

\usepackage[utf8]{inputenc}
\usepackage{color}
\definecolor{note_fontcolor}{rgb}{0.800781, 0.800781, 0.800781}

\usepackage[letterpaper]{geometry}
\usepackage{float}
\usepackage{units}
\usepackage{xfrac} 
\usepackage{amstext}
\usepackage{setspace}
\PassOptionsToPackage{normalem}{ulem}
\usepackage{ulem}

\usepackage[unicode=true,pdfusetitle, bookmarks=true,bookmarksnumbered=false,bookmarksopen=false,
 breaklinks=false,pdfborder={0 0 1},backref=page,colorlinks=true]
 {hyperref}
\hypersetup{
 linkcolor=blue, citecolor=blue}

\title{Aldous-type Spectral Gaps in Generalized Symmetric Groups}
\author{Niv Levhari, Doron Puder}

\newtheorem{theorem}{Theorem}[section]    
\newtheorem{definition}[theorem]{Definition}
\newtheorem{example}[theorem]{Example}
\newtheorem{remark}[theorem]{Remark}
\newtheorem{lemma}[theorem]{Lemma}
\newtheorem{claim}[theorem]{Claim}
\newtheorem{corollary}[theorem]{Corollary}
\newtheorem{proposition}[theorem]{Proposition}
\newtheorem{conjecture}[theorem]{Conjecture}
\numberwithin{equation}{section}

\newcommand{\std}{\mathrm{std}_n}
\newcommand{\Irr}{\mathrm{Irr}}
\newcommand{\Rep}[1]{\text{Rep}\bigl({#1}\bigr)}
\newcommand{\Res}{\text{Res}}
\newcommand{\Irro}{\text{Irr}_0}
\newcommand{\Ind}{\text{Ind}}

\newcommand{\supp}{\mathrm{supp}}
\newcommand{\End}{\mathrm{End}}

\newcommand{\regSn}{\mathrm{Reg}_{S_n}}
\newcommand{\CSn}{\mathbb{C}[S_n]}
\newcommand{\CWn}{\mathbb{C}[W_n]}
\newcommand{\regWn}{\mathrm{Reg}_{W_n}}
\newcommand{\Par}{\mathcal{P}}

\newcommand{\ls}{\lambda_{\min}^*}
\newcommand{\lm}{\lambda_{\min}}

\newcommand{\set}[1]{\left\{{#1}\right\}}
\newcommand{\LL}{\mathcal{L}}

\newcommand{\gl}{\mathrm{GL}}
\newcommand{\triv}{\mathrm{triv}}
\newcommand{\JBG}{J_B^{(G)}}
\newcommand{\LBG}{\LL_\Gamma ^{(G)}}
\newcommand{\vm}{\vec{\mu}}

\newcommand{\U}{\mathrm{U}}
\newcommand{\vio}{v_{i_0}}

\begin{document}

\maketitle

\begin{abstract}
    We prove an analog of Aldous' spectral gap conjecture in the generalized symmetric groups $G\wr S_n$ where $G$ is an arbitrary finite group. 
    Moreover, we show that Caputo's extension of the conjecture to hypergraphs transfers to these groups whenever it holds in the ordinary symmetric group.
\end{abstract}

\tableofcontents{}    

\section{Introduction}\label{sec:intro}

Aldous’ spectral gap conjecture, formulated around 1992 and proved nearly two decades later by Caputo, Liggett and Richthammer \cite{CLR10}, concerns the equality of spectral gaps of the Laplacians of two fundamental Markov processes on a finite graph on $n\ge2$ vertices. 
The first process is the interchange process on a graph $\Gamma$, in which $n$ distinct
balls are placed at the vertices of $\Gamma$, one ball at every vertex. At every step of the process, one picks an edge of $\Gamma$ uniformly at random, and interchanges the two balls at the endpoints of the
chosen edge. The second process is a random walk with a single ball, which is located at one of the vertices. At each step, one
picks an edge of $\Gamma$ uniformly at random as before. If the ball sits at an endpoint of this edge, it is moved to the other endpoint, and otherwise it stays in its place. Although the first process has $n!$ states -- significantly more than the mere $n$ states of the second -- Aldous conjectured that both processes have identical spectral gaps. In \cite{CLR10} this conjecture was proven not only in this scenario but more generally for \textit{weighted} graphs.

This result has an equivalent formulation in the language of the representation theory of the symmetric group $S_n$, where the edges of the graph correspond to transpositions and the processes to certain representations of the group. (This was noted in \cite{Cesi2016Octopus}). Such a formulation suggests the following natural question: does this phenomenon have parallels in other groups? We should stress that it is not a priori clear what the precise generalization should be. In particular, what the analogs of transpositions may be.

This paper focuses on an analog in generalized symmetric groups. 
Throughout, fix a finite group $G$ and $n\ge 2$ and denote
\begin{equation} \label{eq:def of Wn}
    W_n:=G\wr S_n=G^n\rtimes S_n.
\end{equation}
One may think of $W_n$ as the group of \(n\times n\) monomial\footnote{A monomial matrix is a matrix having exactly one non-zero entry in every row and in every column.} matrices whose non-zero entries lie in \(G\) -- see Section \ref{subsec:W_n} for more details. 

Several prior works \cite{CesiWeylB,Ghosh23,AlonGhosh25} suggested analogs of Aldous' conjecture in several generalized symmetric groups with certain analogs for transpositions (as we elaborate in Subsection~\ref{subsec:related-results}). We present here a different generalization, inspired by the work of~\cite{AP26} on unitary groups. In our generalization, a very clean statement, analogous to Aldous' conjecture, holds. Moreover, we show that Caputo's extension of Aldous' conjecture to hypergraphs, while still open, holds in $W_n$ (independently of the finite group $G$) for a given hypergraph, if and only if it holds in the original setting of $S_n$ with the same hypergraph.

\subsection{The Aldous-Caputo conjecture in $S_n$}
We proceed to formally define the setting in the original problem in $S_n$. Although the original conjecture of Aldous considered graphs, the setting we present here is in the generality of hypergraphs which corresponds to the extension by Caputo. Let $\Gamma=([n],c)$ be a weighted \textbf{hypergraph} on $n$ vertices labelled $1,\ldots,n$ (we use the notation $[n]:=\{1,\ldots,n\}$), given by the assignment of a non-negative weight $c_B\in\mathbb{R}_{\ge0}$ to any hyperedge $B\subseteq [n]$. 

At every step of the interchange process or of the random walk, one picks a hyperedge $B\subseteq[n]$ with rate given by its weight $c_B$, and then chooses one of the $|B|!$ permutations in $S_B$ uniformly at random, where 
\[
    S_B:=\set{\sigma\in S_n\,\mid\,\supp\left(\sigma\right)\subseteq B}.
\]
In the interchange process we permute the $|B|$ balls sitting at the vertices of the selected hyperedge $B$ according to the chosen $\sigma\in S_B$. In the random walk, if the ball sits at vertex $i$, we move it to $\sigma(i)$.

Namely, we consider the following \textbf{averaging elements} in the group algebra $\CSn$:
\begin{equation} \label{eq:def of J_B}
    J_B\ :=\ \frac{1}{|S_B|}\sum_{\pi\in S_B}\pi,
\end{equation}
and the \textbf{hypergraph Laplacian element}\footnote{We use $e$ for the unit of a group.}
\begin{equation}\label{eq:def of laplacian}
    \LL_\Gamma\ :=\ \sum_{B\subseteq[n]} c_B\,\left(e-J_B\right)\in\CSn.
\end{equation}

Let $\tau\colon S_n\to \gl_d(\mathbb{C})$ be a finite-dimensional representation of $S_n$, where $\gl_d$ stands for invertible $d\times d$ matrices. The eigenvalues of $\tau(\LL_\Gamma)$ are real non-negative (see Corollary~\ref{cor: LBG is PSD}), and we define 
\begin{equation}\label{eq: lambda min}
    \lm(\Gamma,\tau) := \mathrm{the~smallest~ eigenvalue~of~}\tau\left(\LL_\Gamma\right).
\end{equation}
Each finite-dimensional representation of $S_n$ is uniquely decomposed (up to isomorphism) to irreducible ones $\tau\cong\oplus_i\tau_i$. The spectrum of $\tau(\LL_\Gamma)$ is the union of the spectra of $\tau_i(\LL_\Gamma)$.\footnote{\label{footnote: rep decomposition}With a suitable change of basis, $\tau(\LL_\Gamma)$ is a block-diagonal matrix, with each block corresponding to $\tau'(\LL_\Gamma)$ for some irreducible $\tau'$ in the decomposition of $\tau$.} We denote by $\ls(\Gamma,\tau)$ the smallest eigenvalue which is \textit{not} associated to the trivial summands in this decomposition. Namely, let $\tau\cong \tau'\oplus k\cdot\triv$ for some $k\in\mathbb{Z}_{\ge0}$, where $\tau'$ has no trivial component. Then
\begin{equation}\label{eq: lambda star}
    \ls(\Gamma,\tau) := \mathrm{smallest~ eigenvalue~of~}\tau'\left(\LL_\Gamma\right).
\end{equation}
(If $\tau$ has only trivial components, then $\ls(\Gamma,\tau)=\infty$.)

Let $\std$ denote the standard $S_n$-representation, that is, the non-trivial irreducible $(n-1)$-dimensional representation contained in the permutation representation of $S_n$ acting on $[n]$, and let $\mathrm{Reg}_{S_n}$ be the left-regular representation of $S_n$ (its action by multiplication from the left on $\mathbb{C}[S_n]$). The regular representation contains all of the irreducible representations of $S_n$ in its decomposition, so the spectrum of $\regSn(\LL_\Gamma)$ contains the spectrum of $\std(\LL_\Gamma)$ (as multisets). In particular, for any hypergraph $\Gamma$ as above, $\ls(\Gamma,\regSn)\le \lm(\Gamma,\std)$. Aldous' conjecture states that the two are equal when $\Gamma$ is a graph:

\begin{theorem}[\cite{CLR10}, Aldous' Conjecture]\label{thm:CLR}
    For all graphs\footnote{The hypergraph $\Gamma$ is a graph when $c$ is supported on pairs.} \(\Gamma=([n],c)\)
    \[
        \ls\left(\Gamma,\regSn\right) = \lm\left(\Gamma,\std\right).
    \]
\end{theorem}

Caputo conjectured that the same holds for any hypergraph:

\begin{conjecture}[Caputo Conjecture]\label{conj:Caputo}
    For all hypergraphs $\Gamma=([n],c)$,
    \[
    \ls\left(\Gamma,\regSn\right) = \lm\left(\Gamma,\std\right)
    \]
\end{conjecture}

This conjecture appeared in~\cite[Conj.~3]{Piras2010},~\cite[p.~301]{Cesi2016Octopus}, and~\cite[p.~78]{AldousCaputoDurrettHolroydJungPuha2021Liggett}. Non-trivial special cases are proven in \cite{ALON_KOZMA_PUDER_2025} and \cite[Theorem~1.8]{bristiel2024entropy}.

Denote by ${\Par}_n$ the set of partitions of $n$ (equivalently, Young diagrams with $n$ boxes), and let $\Par=\sqcup_{n\ge0}{\Par}_n$ be the set of finite integer partitions.
Equivalently, recalling the indexing of irreducible representations of $S_n$ by the partitions in $\Par_n$, and denoting the representation corresponding to $\mu\vdash n$ by $\tau_\mu$, Theorem~\ref{thm:CLR} and Conjecture~\ref{conj:Caputo} state that the minimal non-trivial eigenvalue of the Laplacian appears in $\tau_{(n-1,1)}$, the irreducible representation corresponding to the Young diagram with $n-1$ boxes in the first row.

\subsection{Hypergraph distributions in $W_n=G\wr S_n$}

Recall that \(G\) is a fixed but arbitrary finite group. We explore the Aldous phenomena in the ``$G$-generalized" symmetric groups, the wreath products \(W_n=G\wr S_n\). We will shortly associate a ``hypergraph Laplacian element" in the group algebra $\mathbb{C}[W_n]$ to every hypergraph $\Gamma=([n],c)$, analogously to \eqref{eq:def of laplacian}. Our results show that for any $\Gamma$, the associated spectral gap of the regular representation of $W_n$ is identical to the associated spectral gap of the regular representation of $S_n$, and that if Caputo's conjecture holds for $\Gamma$, then an analogous statement holds for the same hypergraph in $W_n$. In particular, the analog of Theorem \ref{thm:CLR} is true in $W_n$.

\medskip
Recall that one may think of the group \(W_n=G\wr S_n=G^n\rtimes S_n\) as the group of \(n\times n\) monomial matrices whose non-zero entries lie in \(G\). For \(B\subseteq[n]\) define \(W_B\leq W_n\) to be the subgroup consisting of matrices that are identical to the identity matrix outside the $|B|\times|B|$ minor given by $B$. Namely, these are matrices which have the identity element $e_G$ of $G$ at every diagonal entry $(i,i)$ with $i\notin B$.

\begin{example}
    Let \(n=5\) and \(B=\{1,2,4\}\), then
    \[
        W_B = \left\{ \begin{pmatrix}
        
        * & * &  & * &  \\
        * & * &  & * &  \\
           &  & e_G &  &  \\
        * & * &  & * &  \\
           &  &  &  & e_G
          \end{pmatrix} \right\}\leq W_n.
    \]

\end{example}

\medskip
\noindent We can now generalize our definitions to the wreath product.

\begin{definition}[Wreath product hypergraph Laplacian]\label{def:Wn-hyper-graph} 
    For $B\subseteq[n]$ define
    \begin{equation}\label{eq: def of wreath J_B}
        \JBG\ :=\ \frac{1}{|W_B|}\sum_{g\in W_B}g \ \in\ \mathbb{C}[W_n].
    \end{equation}
    For a weighted hypergraph $\Gamma=([n],c)$ define the wreath product hypergraph Laplacian by
    \begin{equation}\label{eq: def of wreath laplacian}
        \LBG\ :=\ \sum_{B\subseteq[n]} c_B\,\left(e-\JBG\right)\ \in\ \mathbb{C}[W_n].  
    \end{equation}
\end{definition}

\subsection{Main results}\label{subsec: main results}
Let \(\Irr(G)\) be the set of (equivalence classes of) irreducible representations of \(G\). 
As we expand on in Section~\ref{sec:prelims}, irreducible \(W_n\)-modules are indexed by functions 
\[
    \vec{\mu}:\Irr \left(G\right)\to \Par ~~~\text{ s.t. } \sum_{\theta\in\Irr\left(G\right)}|\vec{\mu}\left(\theta\right)|=n,
\] 
which we encourage the reader to think of as $\Irr (G)$--indexed Young diagrams. The notation $|\vec\mu(\theta)|$ refers to the size of the partition (or Young diagram) corresponding to $\theta$, namely $|\vec\mu(\theta)|=m$ if and only if $\mu(\theta)\in\Par_m$. We also write $|\vec\mu|$ for the sum $\sum_{\theta\in\Irr(G)}|\vec{\mu}(\theta)|$. We sometimes refer to these functions as \textbf{$\Irr(G)$-indexed multi-partitions of order $n$}.

We denote by \(\rho_{\vec{\mu}}\) the $W_n$-irreducible representation\footnote{We use $\tau$ for representations of $S_n$ and $\rho$ for those of $W_n$ to avoid confusion, unless a representation of $W_n$ is a lift of a representation from $S_n$, and then we may refer to it by $\tilde\tau$.} corresponding to such 
\(\vec{\mu}\). 
The \textbf{support} of $\vec\mu$ is 
\begin{equation}\label{eq: def of support of multipartition}
    \supp (\vec \mu) := \set{\theta\in\Irr \left(G\right) :\vec\mu\left(\theta\right)\neq\emptyset}.
\end{equation}

We denote the set of $W_n$ irreducible representations supported on the trivial irreducible representation of $G$ by
\begin{equation}\label{eq: def of triv supported irreps}
    \Irro \left(W_n\right) := \set{\rho_{\vec\mu}:\supp\left(\vec\mu\right)=\set{\mathrm{triv}_G}}.
\end{equation}
\noindent
Since there is an epimorphism $\pi\colon W_n\to S_n$ (see~\eqref{eq: def of pi}), every (irreducible) representation $\tau$ of $S_n$ \textbf{lifts} to an (irreducible) representation of $W_n$ given by $\tilde{\tau}:=\tau\circ\pi$.
The following claim is straightforward from the construction of the irreducible representations of $W_n$, mentioned in Subsection~\ref{subsec:Wn-indexing}:
\begin{claim}\label{claim:triv}
    The representations in $\Irro(W_n)$ are the lifts of the representations in $\Irr(S_n)$.
    Specifically, if $\mu\vdash n$ and $\tau_\mu\in\Irr(S_n)$ is the corresponding representation, it lifts to $\rho_{\vm}\in\Irro(W_n)$ with $\vm(\triv_G)=\mu$, namely, $\tilde{\tau_\mu}=\rho_{\vm}$. Moreover, the hypergraph Laplacian elements are equal in such representations:
    \[
        \tau_{\mu}\left(\LL_\Gamma\right) = \rho_{\vm}\left(\LBG\right)
    \]
\end{claim}
(In the equality, if $\tau_\mu\colon S_n\to\mathrm{GL}(V)$ for some finite-dimensional complex vector space $V$, we think of $\rho_{\vm}=\tilde{\tau_\mu}\colon W_n\to\mathrm{GL}(V)$ as a representation over the \textit{same} vector space $V$).
We prove this claim in Subsections~\ref{subsec:triv G^n reps} and~\ref{subsec: average properties}.

We reuse the notations $\lm(\Gamma,\rho)$ from~\eqref{eq: lambda min} and $\ls(\Gamma,\rho)$ from~\eqref{eq: lambda star} for a representation $\rho$ of $W_n$, when now the corresponding hypergraph Laplacian element is \eqref{eq: def of wreath laplacian}. For example, $\lm(\Gamma,\rho)$ is the smallest eigenvalue of $\rho(\LBG)$. Since the representation $\phi$ in the pair $(\Gamma,\phi)$ determines the ambient group, no confusion should arise with \eqref{eq: lambda min} and \eqref{eq: lambda star}.

As we mentioned in Footnote \ref{footnote: rep decomposition}, each finite-dimensional representation of $W_n$ is uniquely decomposed to irreducible ones $\rho\cong\oplus_i\rho_i$, and the spectrum of $\rho(\LBG)$ is the union of spectra of $\rho_i(\LBG)$. Our main result is that the smallest non-trivial eigenvalue of the regular representation of $W_n$, denoted $\regWn$, is always associated to some representation lifted from $S_n$.

\begin{theorem}[Main result]\label{thm:main}
    For every finite \(G\) and every hypergraph \(\Gamma=([n],c)\) with $n\ge2$ and a non-negative weight function $c$,
    \[
        \ls\left(\Gamma, \regWn\right) = \ls\left(\Gamma, \regSn \right).
    \]
\end{theorem}
\medskip

\noindent In fact, we prove a slightly stronger statement: A vertex $i\in[n]$ in a hypergraph $\Gamma$ is called \textit{almost-isolated}\label{page:almost-isolated} if $c_B=0$ whenever $i\in B\subsetneq[n]$ (we allow $c_{[n]}$ to be positive). If a hypergraph $\Gamma$ has no almost-isolated vertex, then for any $\rho\in\Irr (W_n)\setminus\Irro(W_n)$,
    \[
    \lm\left(\Gamma,\rho\right)> \ls\left(\Gamma, \regWn\right)
    \]
with \textit{strict} inequality.

\begin{corollary}[Reduction to Caputo’s conjecture]\label{cor:transfer}
    Let \(G\) and $\Gamma$ be as in Theorem~\ref{thm:main}. 
    If Conjecture~\ref{conj:Caputo} holds for $\Gamma$, then
    \[
    \ls\left(\Gamma,\regWn\right)=\ls\left(\Gamma,\regSn\right)=\lm\left(\Gamma,\std\right),
    \]
    where, as above, $\std$ denotes the standard $n-1$-dimensional irreducible representation of $S_n$.
\end{corollary}
In particular, every verified instance of Caputo’s conjecture for \(S_n\) extends to \(W_n=G\wr S_n\):
\begin{corollary}[Verified instances]\label{cor:explicit-cases}
    Let \(G\) and $\Gamma$ be as in Theorem~\ref{thm:main}. Suppose one of the following holds:
    \begin{enumerate}
        \item \label{enu:concrete case 1} \textbf{Pairs and Singletons:} The weight function $c$ is supported on $|B|\le 2$.
        \item \textbf{$(n-1)$-tuples:} \label{enu: n-1}The weight function $c$ is supported on $|B|\geq n-1$.
        \item \textbf{Mean-field:} \label{enu:meanfield}The weights \((c_B)\) are such that $c_B$ depends only on \(|B|\), that is, for any $|A|=|B|$, $c_A=c_B$.
        \item \textbf{Tree-like:} \label{enu:tree-like}The weights \((c_B)\) are supported on hyperedges that can be removed as leaves inductively in a tree-inspired manner (see~\cite{ALON_KOZMA_PUDER_2025}, Section 7). 
        \item \textbf{\cite[Theorem~1.3]{ALON_KOZMA_PUDER_2025}-hypergraphs:}\label{enu: akp hypergraphs} There is some $B_0\subseteq[n]$ such that $c$ is supported on subsets $B$ containing $B_0$ with $|B\setminus B_0|\le2$.\footnote{So Case \ref{enu:concrete case 1} is the special case with $B_0=\emptyset$ of Case \ref{enu: akp hypergraphs}.}
    \end{enumerate}
    Then the analog of Caputo’s conjecture in $W_n$ holds for $\Gamma$, that is,
    \[
    \ls\left(\Gamma,\regWn\right)=\ls\left(\Gamma,\regSn\right)=\lm\left(\Gamma,\std\right).
    \]
\end{corollary}

Indeed, for $S_n$, Caputo’s conjecture is known for pairs~\cite{CLR10} and if $|B|\le1$ then $e-J_B=0$ in $\CSn$, so the weights of singletons and of the empty set do not alter the Laplacians.
When the weight function is supported on subsets $B$ of size $\ge n-1$, the spectrum of $\tau_\mu(\LL_\Gamma)$ is identically $\sum_B c_B$ whenever $\mu\ne(n),(n-1,1)$ (this is an easy consequence of the branching rule). Hence the smallest non-trivial eigenvalue must come from $\tau_{(n-1,1)}$ --  see~\cite[Section~7]{ALON_KOZMA_PUDER_2025} for more details. 
The mean-field Case~\ref{enu:meanfield} in $S_n$ is~\cite[Theorem~1.8]{bristiel2024entropy}, the tree-like Case \ref{enu:tree-like} in $S_n$ is proven in~\cite[Section~7]{ALON_KOZMA_PUDER_2025}, and the last case is~\cite[Theorem 1.3]{ALON_KOZMA_PUDER_2025}. Corollary~\ref{cor:transfer} immediately implies the corresponding statements for $W_n$.

\begin{remark}\label{rem:tups}
    As mentioned above, Case~\ref{enu: n-1} of weight function supported on subsets of size $\ge n-1$ has a short and simple argument in $S_n$. However, this argument does not extend directly to $W_n$: there are irreducible representations $\rho$, other than $\tilde{\tau}_{(n-1,1)}$, with a non-trivial spectrum of $\rho(\LBG)$. In fact, these are precisely the irreducible representations of $W_n$ corresponding to $\vm$ with $\vm(\triv_G)=(n-1,1)$ (this is $\tilde{\tau}_{(n-1,1)}$) and those with $\vm(\triv_G)=(n-1)$ and $\vm(\theta)=(1)$ for some $\triv_G\ne\theta\in\Irr(G)$. Hence, Corollary~\ref{cor:transfer} is needed even for this case, which is fairly straightforward in the case of $S_n$.

    A similar issue arises regarding tree-like hypergraphs in Case \ref{enu:tree-like}. The proof in \cite{ALON_KOZMA_PUDER_2025} for $S_n$, while simple, does not transfer verbatim to $W_n$ either. However, Corollary~\ref{cor:transfer} proves it regardless.
\end{remark}

\subsection{Related works}\label{subsec:related-results}
Previous results in the wreath–product setting were hitherto only obtained in settings closer to the original \textit{graph} settings of Aldous, not in the more general hypergraph setting that we present. These results indicated that the nature of an Aldous phenomenon, when it exists, is highly dependent on the precise family of measures one considers. 

\paragraph{Notation.} The results presented in this subsection all share a similar structure, so we will use the following shared notations for all of them. 
Fix a finite group $G$. 
For a fixed $\triv_G\ne\chi\in\Irr(G)$, denote by $\rho_{\triv_G\mapsto \nu_1;\chi\mapsto \nu_2}$ the irreducible representation of $W_n=G\wr S_n$ corresponding to $\vm$ which maps $\triv_G\in\Irr(G)$ to the partition $\nu_1\in\Par$ and $\chi\in\Irr(G)$ to the partition $\nu_2\in\Par$ with $|\nu_1|+|\nu_2|=n$.
For a subset $B\subseteq [n]$ and an element $g\in G$, define $s_{g,B}\in G^n \le W_n$ to be the diagonal matrix with $g$ in the coordinates of $B$, and $e_G$ in the remaining coordinates ($[n]\setminus B$).
For a given $j\neq k\in[n]$, $(j~k)$ is the transposition (an ordinary permutation matrix).
For any $i,j,k\in[n]$, any $g\in G$ and any $B\subseteq[n]$, $a_{i},b_{jk},\alpha_B,c_g\in\mathbb{R}_{\ge0}$ are all non-negative weights, and $c_g=c_{g^{-1}}$.

\medskip

Cesi’s work in \cite{CesiWeylB} on the signed symmetric group \(C_{2}\wr S_{n}\) (here $C_2=\{\pm1\}$ is the group of order 2), establishes an exact spectral-gap identity for certain Cayley graphs. 
The measures he explores are of the form
\[
    \mathcal{A} \ := \ \sum_{i=1}^n a_i s_{\left(-1\right),\{i\}} + \sum_{1\le j<k\le n} b_{jk} \left(j~k\right).
\]

In Cesi's analysis, the representations governing the minimal eigenvalue of the Laplacian contain the lift of the standard representation as in Aldous' Conjecture, but they are not restricted to that representation alone: the gap arises from the permutation representation corresponding to the action of $C_2\wr S_n$ on $C_2\times[n]$, which decomposes as 
\[
\tilde\tau_{\left(n-1,1\right)}\oplus\rho_{\triv_{C_2}\mapsto\left(n-1\right);\mathrm{sign}\mapsto\left(1\right)}\oplus\triv_{C_2\wr S_n}.
\]
Moreover, both irreducible representations are needed, as is shown in~\cite[Section~5]{CesiWeylB}.

A similar phenomenon occurs in Ghosh’s Aldous-type theorem for the same class of groups we consider in the current paper, namely, for general finite $G$ and the wreath product \(W_n=G\wr S_{n}\) \cite{Ghosh23}. The measures considered there are of the form\footnote{In Ghosh's paper, the required assumption is $a_i>0$. We present the statement here with $a_i\ge0$, where the only difference is that the gap could be $0$ if one is not careful enough with the choice of $a_i$'s.}
\[
    \mathcal{B} \ := \ \sum_{i=1}^n a_i \sum_{g\in G} c_g s_{g,\set{i}} + \sum_{1\le j<k\le n} b_{jk}\left(j~k\right).
\]

The spectral gap in this setting is again controlled by a permutation representation, again arising from the natural action of $W_n$ on $G\times[n]$, given by\footnote{We define the notation $(\textbf{g};\sigma)$ explicitly in Section~\ref{sec:prelims}.} $(\textbf{g};\sigma).(h,i)=(\textbf{g}_{\sigma(i)}\cdot h,\sigma(i))$. This representation decomposes as the direct sum 
\[
    \triv_{W_n} \oplus \tilde\tau_{\left(n-1,1\right)} \oplus \bigoplus_{\triv_G\ne\chi\in\Irr \left(G\right)} \left(\dim(\chi)\right)\cdot\rho_{\triv_G\mapsto\left(n-1\right);\chi\mapsto\left(1\right)},
\]
as shown in \cite[Theorem 3.11]{Ghosh23}.\footnote{This can also be seen using the fact that the stabilizer of $(e_G,n)\in G\times[n]$ is $G\wr S_{n-1}$, induction from $\triv_{G\wr S_{n-1}}$ and the branching rule.}
This generalizes Cesi's work from $C_2$ to a general finite group $G$. (In this case, it was not shown that the spectral gap could arise from any irreducible representations in this decomposition.)

A recent preprint of Alon and Ghosh \cite{AlonGhosh25} gives another glimpse into the phenomena of \(C_{2}\wr S_{n}\), where in this case they allow the $s_{g,B}$ elements to take arbitrary subsets $B$, and not only singletons. These measures give a different generalization of Cesi's work, and are of the form
\[
    \mathcal{C} := \sum_{B\subseteq[n]}\alpha_B s_{\left(-1\right),B} +\sum_{1\le j<k\le n}b_{jk}\left(j~k\right).
\]

Yet again, the dominating representation contains the standard representation but is not equal to it. In this case, the dominating representation is in the collection
\[
\mathcal F_n=\set{\tilde\tau_{\left(n-1,1\right)}}\cup\set{\rho_{\triv_{C_2}\mapsto\left(n-m\right); \mathrm{sign}\mapsto \left(m\right)}\mid m\in[n]}.
\]
Note that this collection contains the irreducible representation from Cesi's work (taking $m=1$). They also show that any of these irreducible representations may induce the spectral gap.

Since the measures studied in \cite{CesiWeylB,Ghosh23,AlonGhosh25} differ from the ones considered here, their recurring pattern is more complex: it is attained not by the pure standard representation of \(S_{n}\), but by a higher-dimensional \(W_{n}\)-representation containing it. Our setting of the hypergraph measures (see~\eqref{eq: def of wreath laplacian}) leads to a cleaner statement, which we see as desirable. 
\medskip

Finally, as mentioned above, the setup of the current paper is inspired by the work \cite{AP26} on the unitary group $\U(n)$. That paper defines hypergraph measures on $\U(n)$, and the question is whether the spectral gap of the corresponding operator in the regular representation of $\U(n)$ is always obtained in the same few irreducible representations. Similarly, in line with our findings here, it is shown in \cite[Theorem 1.10]{AP26} that the spectral gap must always be obtained in some representation admitting a non-trivial subspace which is invariant under the diagonal subgroup of $\U(n)$. Discussion of further possible extensions of the Aldous phenomenon within the symmetric groups or in more general Coxeter groups can be found in \cite[$\S$5]{Cesi2016Octopus} and in \cite{PP18}.

\subsection{Notation and outline}
\paragraph{Notation.}
Throughout, $G$ is a finite group and $2\le n\in\mathbb{N}$. We write
\[
W_n := G\wr S_n = G^n \rtimes S_n
\]
as in~\eqref{eq:def of Wn}, with the action of $S_n$ permuting the $G^n$ coordinates in the usual way. We give an exact description of the elements in~\eqref{eq: def mat W_n} below. For a weighted\footnote{Our weights are always assumed to be non-negative.} hypergraph $\Gamma=([n],c)$, the wreath–product hypergraph Laplacian and its summands
\[
\LBG := \sum_{B\subseteq[n]} c_B\,\left(e-\JBG\right)\in\CWn
\]
are defined in~\eqref{eq: def of wreath J_B}-\eqref{eq: def of wreath laplacian}.

The (isomorphism classes of) irreducible representations of a finite group $G$ are denoted by $\Irr(G)$.

We use $d\cdot V$ for a vector space $V$ (respectively $d\cdot\theta$ for a representation $\theta$) to denote a direct sum of $d$ copies of $V$ (respectively of $\theta$).

Irreducible representations of $S_n$ are indexed by partitions $\mu\vdash n$; we denote the corresponding representation by $\tau_\mu$. The \emph{lift} of $\tau_\mu$ to $W_n$ is
\[
\tilde{\tau_\mu} := \tau_\mu\circ\pi,
\]
where $\pi:W_n\to S_n$ is the canonical projection.

Irreducible representations of $W_n$ are indexed by $\Irr(G)$–indexed multi-partitions $\vec\mu$ of order $n$ (see Subsection~\ref{subsec:Wn-indexing}); the corresponding irreducible representation is denoted $\rho_{\vec\mu}$. When $\supp(\vec\mu)=\{\triv_G\}$, we have $\rho_{\vec\mu}=\tilde{\tau_\mu}$ with $\mu=\vec\mu(\triv_G)$ (by Claim~\ref{claim:triv}).

\paragraph{Outline.}
The paper is organized as follows:
Section~\ref{sec:prelims} gives relevant background on wreath products and their representation theory, proves technical lemmas regarding irreducible representations of $W_n$ supported on $\triv_G$ and regarding the properties of hypergraph measures on $W_n$.
Section~\ref{sec:main} establishes the main result by bounding the Laplacian spectrum from above on an irreducible representation whose restriction to $G^n$ is trivial, and from below (with the same bound) on any irreducible representation whose restriction to $G^n$ is non-trivial.

\subsection*{Acknowledgments}
This work was supported by the European Research Council (ERC) under the European Union’s Horizon 2020 research and innovation programme (grant agreement No 850956), by the Israel Science Foundation, ISF grants 1140/23, as well as by the Kovner Member Fund at the IAS, Princeton.

\section{Preliminaries}\label{sec:prelims}

\subsection{The wreath product $W_n$} \label{subsec:W_n}
Above, we considered elements of $W_n=G\wr S_n=G^n\rtimes S_n$ as $n\times n$ monomial matrices with entries from $G\sqcup\{0\}$. It is also useful to write them as \((\textbf{g};\sigma)\in G^n\rtimes S_n\), with $\textbf{g}=(g_1,\ldots,g_n)\in G^n$ and $\sigma\in S_n$. The corresponding matrix $M_{\textbf{g},\sigma}\in\mathrm{Mat}_{n\times n}(G\sqcup{\set{0}})$ is given by
\begin{equation}\label{eq: def mat W_n}
[M_{\textbf{g},\sigma}]_{r,c}=\begin{cases}
    g_r & \mathrm{if~}r=\sigma\left(c\right) \cr
    0 & \mathrm{otherwise}.
\end{cases}
\end{equation}

The symmetric group $S_n$ embeds into $W_n$, with the embedding map given by $\sigma\mapsto (e_{G^n};\sigma)$.
The product $G^n$ embeds into $W_n$ through $\textbf{g}\mapsto(\textbf{g};e_{S_n})$. When writing the elements as $(\textbf{g};\sigma)$, the multiplication is given by 
\[
    \left(\textbf{g};\sigma\right)\cdot\left(\textbf{h};\tau\right)=\left(\textbf{g}\cdot\left(\textbf{h}\circ\sigma^{-1}\right);\sigma\tau\right),
\]
namely,
\begin{equation}\label{eq:wreath-mul}
    \left(g_1,\ldots,g_n;\sigma\right)\cdot\left(h_1,\ldots,h_n;\tau\right)\ =\ \left(g_1\,h_{\sigma^{-1}\left(1\right)},\ldots,g_n\,h_{\sigma^{-1}\left(n\right)};\ \sigma\tau\right).
\end{equation}
We shall use the following decomposition into factors in the wreath product $W_n$ repeatedly:
\begin{equation}\label{eq:wreath-fact}
    \left(\textbf{g};\sigma\right)=\left(\textbf{g};e_{S_n}\right)\left(e_{G^n};\sigma\right)=\left(e_{G^n};\sigma\right)\left(\textbf{g}\circ\sigma;e_{S_n}\right).
\end{equation}
The projection to $S_n$ is
\begin{equation}\label{eq: def of pi}
    \pi:W_n\to S_n,\qquad \pi\left(\textbf{g};\sigma\right)=\sigma,
\end{equation}
and its kernel is the normal subgroup $G^n\trianglelefteq W_n$.

\subsection{The Irreducible Representations of Wreath Products}\label{subsec:Wn-indexing}
We use the following basic representation-theoretic fact repeatedly: if $A$ and $B$ are finite groups, then every irreducible representation of $A\times B$ is an external tensor product $\alpha\boxtimes\beta$ with $\alpha\in\Irr (A)$ and $\beta\in\Irr (B)$, and $\alpha\boxtimes\beta\cong\alpha'\boxtimes\beta'$ if and only if $\alpha\cong\alpha'$ and $\beta\cong\beta'$.

Applied to $G^n=G\times\cdots\times G$, this fact implies that every irreducible representation $\vartheta$ of $G^n$ is (up to isomorphism) an external tensor product 
\begin{equation}
    \vartheta = \boxtimes_{i=1}^n \theta_i,\qquad\theta_i\in\Irr \left(G\right).
\end{equation} 

\subsubsection*{Indexing the irreducible representations of wreath products by multi-partitions}
In Subsection~\ref{subsec: main results} we mentioned that the irreducible representations of $W_n$ are classified by partition-valued functions $\vm\colon\Irr(G)\to\Par$ with $|\vm|=n$. We show one method to decode this indexing, roughly following \cite[Chapter 2.6]{ceccherini2014representation}. We use this decoding below in the lemmas leading to our main results. 
A slightly different construction of the representations of $W_n$ is given  by Zelevinsky in~\cite{Zelevinsky81}.

Fix a multi-partition $\vm\colon\Irr(G)\to\Par$ with $|\vm|=n$, and assume its support is $\supp(\vm) = \{\theta_1,\ldots,\theta_k\}$.\footnote{\label{footnote: theta order does not matter}The support is defined in~\eqref{eq: def of support of multipartition}. It is easy to see that the chosen order of the support does not matter for the resulting representation due to the induction step below.}  
We now construct a representation $\rho_{\vec\mu}$ corresponding to $\vec{\mu}$. 
Set $m_i  :=   |\vec{\mu}(\theta_i)|$ and $\vec m :=(m_1,\ldots,m_k)$. 
Denote 
\begin{equation}\label{eq: S_vec_m}
S_{\vec m}:=S_{\set{1,\ldots,m_1}}\times S_{\set{m_1+1,\ldots,m_1+m_2}} \times\cdots\times S_{\set{n-m_k+1,\ldots,n}}\le S_n,
\end{equation}
so $S_{\vec m}\cong S_{m_1}\times S_{m_2} \times\cdots\times S_{m_k}$.
Consider the irreducible $G^n$–representation
\begin{equation} \label{eq:vartheta}
\vartheta\ :=\ \underbrace{\theta_1\boxtimes\cdots\boxtimes\theta_1}_{m_1~\mathrm{times}}\boxtimes\cdots\boxtimes\underbrace{\theta_k\boxtimes\cdots\boxtimes\theta_k}_{m_k~\mathrm{times}}   
\end{equation}

and denote by $f:[n]\to\Irr(G)$ the function that maps $t\in[n]$ to the irreducible representation of $G$ corresponding to the $t^{\mathrm{th}}$ coordinate of $G^n$.
Note that $f=f\circ\sigma$ for any $\sigma\in S_{\vec m}$.
Let $V_i$ denote the vector space of the representation $\theta_i$, and 
let $$V_{\vartheta}=V_1^{\otimes m_1}\otimes \cdots\otimes V_k^{\otimes m_k}$$ be the vector space of the representation $\vartheta$. For any $\sigma\in S_{\vec m}$, we define a linear bijection $\hat\sigma:V_{\vartheta}\to V_{\vartheta}$ by permuting tensor coordinates according to $\sigma$.

Denote $W_{\vec m}:=G^n\rtimes S_{\vec m}\le W_n$ and define 
a $W_{\vec m}$-representation $\varrho$ on the same vector space $V_{\vartheta}$ by letting $G^n$ act component-wise and $S_{\vec m}$ act by the bijections $\hat{\sigma}$. One needs to check that this is indeed a homomorphism: see, for example,~\cite[Lemma 2.4.3]{ceccherini2014representation} for this computation.

For $\mu\vdash m$ let $\tau_\mu$ be the corresponding irreducible representation of $S_m$. 
Recall that $\vec{\mu}(\theta_i)\vdash m_i$ for any $i\in[k]$, and set
\[
\tau_{\vec\mu}\ :=\ \boxtimes_{i=1}^k \tau_{\vec\mu\left(\theta_i\right)}\ \in\ \Irr\left({S_{\vec m}}\right).
\]
By lifting along the projection $W_{\vec m}\twoheadrightarrow S_{\vec m}$, $\tilde\tau_{\vec\mu}$ is a representation of $W_{\vec m}$. We have now defined two representations of $W_{\vec m}$: $\varrho$ and $\tilde\tau_{\vec\mu}$. Define the $W_n$-representation corresponding to $\vm$ by
\begin{equation} \label{eq:define rho_vm by induction of representations}
    \rho_{\vec\mu}\ :=\ \Ind_{W_{\vec m}}^{W_n}\!\left(\varrho\otimes \tilde\tau_{\vec\mu}\right).
\end{equation}

\begin{theorem}[Classification of $\Irr(W_n)$, {\cite[Theorem~2.6.1]{ceccherini2014representation}}]\label{thm:Wn-classification}
Fix $n$ and a finite group $G$. For each multi-partition $\vm:\Irr(G)\to\Par$ with $|\vm|=n$,
define $W_{\vec m}\le W_n$, $\vartheta\in\Irr(G^n)$, $\varrho\in\Irr(W_{\vec m})$, and $\tau_{\vm}\in\Irr(S_{\vec m})$ as above, and set
\[
\rho_{\vm}:=\Ind_{W_{\vec m}}^{W_n}\left(\varrho\otimes \tilde\tau_{\vm}\right).
\]
Then $\rho_{\vm}$ is irreducible, $\rho_{\vm}\not\simeq\rho_{\vec\nu}$ for $\vm\neq\vec\nu$, and every irreducible
$W_n$--representation is isomorphic to $\rho_{\vm}$ for a unique $\vm$.
\end{theorem}
As we mentioned in Footnote~\ref{footnote: theta order does not matter}, changing the order we set for the support yields isomorphic irreducible representations of $W_n$. 

\subsection{$W_n$ representations supported on $\triv_{G}$}\label{subsec:triv G^n reps}
\begin{lemma}\label{lem:restriction-Gn-type}
Let $\vm:\Irr(G)\to\Par$ be a multi-partition with $|\vm|=n$. Define $\rho_{\vm}\in\Irr(W_n)$, $\vartheta\in\Irr(G^n)$, $\varrho\in\Irr(W_{\vec m})$, and $\tau_{\vm}\in\Irr(S_{\vec m})$ as above. The restriction of $\rho_{\vm}$ to $G^n$ satisfies
\[
    \Res^{W_n}_{G^n}\rho_{\vm} \cong \bigoplus_{\theta\in\Irr\left(G^n\right):\theta\sim\vartheta}\left(\dim \tau_{\vm}\right)\cdot\theta
\]
where $\theta\sim\vartheta$ if each irreducible representation of $G$ appears in the same number of coordinates in both representations $\theta$ and $\vartheta$.
In particular, the restriction of $\rho_{\vm}$ to $G^n$ is trivial if and only if $\supp(\vm)=\{\triv_G\}$.
\end{lemma}

Note that there are precisely $\frac{|S_n|}{|S_{\vec m}|}=\frac{|W_n|}{|W_{\vec m}|}$ irreducible representations $\theta\in\Irr(G^n)$ with $\theta\sim\vartheta$.\footnote{By the orbit-stabilizer theorem, as $S_n$ acts on them by permuting the coordinates, and $S_{\vec m}$ stabilizes $\vartheta$.}

\begin{proof}
Denote $d:=\dim(\tau_{\vm})$. First, the restriction of $\tau_{\vm}$ to $G^n$ is trivial, and the restriction of $\varrho$ to $G^n$ is $\vartheta$. Hence, 
\[
\Res^{W_{\vec m}}_{G^n}\left(\varrho\otimes \tilde\tau_{\vm}\right)\ \cong\ \vartheta\otimes \left(d\cdot\triv\right)\cong\ d\cdot\vartheta.
\]

Fix a set of representatives $\mathcal{T}\subset S_n$ for the left cosets $S_n/S_{\vec m}$ and view it inside $W_n$ via
$t\mapsto (e_{G^n};t)$: then $\mathcal{T}$ is also a transversal for the left cosets of $W_{\vec m}$ inside $W_n$. Define $V_{\tau_{\vm}}$ as the vector space corresponding to $\tau_{\vm}$ (and $\tilde\tau_{\vm}$). By the definition of induction of representations, the vector space corresponding to $\rho_{\vm}$ is
$$V_{\rho_{\vm}}=\bigoplus_{t\in\mathcal{T}}t \left(V_{\vartheta}\otimes V_{\tau_{\vm}}\right),$$
where $t (V_{\vartheta}\otimes V_{\tau_{\vm}})$ is a copy of $V_{\vartheta}\otimes V_{\tau_{\vm}}$ and $w\in W_n$ acts by 
\begin{equation}\label{eq: action of W_n on irreps}
w.\left(t\left(v\otimes u\right)\right)=t'\left( w'. v\otimes w'.u\right)
\end{equation}
when $w t = t' w'$ with $w'\in W_{\vec m},t'\in\mathcal{T}$.
Specifically, for $\textbf{g}\in G^n\le W_n$ we have $\textbf{g}\cdot t=(\textbf{g};e_{S_n}) \cdot (e_{G^n}; t)=(e_{G^n};t)\cdot (\textbf{g}\circ t;e_{S_n})=t\cdot (\textbf{g}\circ t)$ by \eqref{eq:wreath-mul}. 
So 
\[
\textbf{g}.\left(t\left( v\otimes u\right)\right)=t \left(\left(\textbf{g}\circ t\right). v\otimes u\right).
\]
So each summand $t (V_{\vartheta}\otimes V_{\tau_{\vm}})$ is invariant under $G^n$, and isomorphic to the $G^n$-representation ${}^{t}\!\vartheta\otimes \left(d\cdot\triv\right)$, where ${}^{t}\!\vartheta$ denotes the $G^n$--representation obtained from $\vartheta$ by permuting tensor coordinates by $t$ as above, 
that is, ${}^{t}\!\vartheta(g)=\vartheta(g\circ t)$. Therefore,
\[
\Res^{W_n}_{G^n}\rho_{\vm}\ \cong\ \bigoplus_{t\in\mathcal{T}} \left(d\cdot{}^{t}\!\vartheta\right).
\]
We are done as the representations ${}^{t}\!\vartheta$ are precisely all the representations $\theta\in\Irr(G^n)$ with $\theta\sim\vartheta$.
\end{proof}

We can now prove the first part of Claim~\ref{claim:triv},
\begin{corollary}\label{cor:trivial-color=lift}
    Let $\vm:\Irr(G)\to\Par$ satisfy $|\vm|=n$. Then the following are equivalent:
    \begin{enumerate}
        \item \label{enu:supp=triv} $\rho_{\vm}\in\Irr_0(W_n)$, namely, $\supp(\vm)=\set{\triv_G}$.
        \item \label{enu:restriction=triv} The restriction of $\rho_{\vm}$ to $G^n$ is trivial.
        \item \label{enu:lift from Sn} $\rho_{\vm}\cong \tilde\tau_{\vm(\triv_G)}$.
    \end{enumerate}
\end{corollary}

\begin{proof}
    The equivalence \eqref{enu:supp=triv}$\Longleftrightarrow$\eqref{enu:restriction=triv} is part of Lemma~\ref{lem:restriction-Gn-type}.     The \eqref{enu:lift from Sn} $\Longrightarrow$ \eqref{enu:restriction=triv} implication is by definition as $G^n$ is the kernel of the projection $\pi:W_n \twoheadrightarrow S_n$ and $\tilde{\tau}_{\vm(\triv_G)}=\tau_{\vm(\triv_G)}\circ\pi$.
    
    Finally, we show that \eqref{enu:supp=triv} $\Longrightarrow$ \eqref{enu:lift from Sn}. By the discussion above, in this case $W_{\vec m}=W_n$ and the induction defining $\rho_{\vm}$ in \eqref{eq:define rho_vm by induction of representations} is trivial. In addition, the action of $S_{\vec m}=S_n$ on $\varrho$ is trivial (as the entire vector space $V_\vartheta$ is one-dimensional) and the action of $G^n$ on $\varrho$ is also trivial by definition. So $W_n$ acts trivially on $\varrho$. Thus,
    \[
    \rho_{\vm}=\Ind_{W_n}^{W_n}\!\left(\varrho\otimes \tilde\tau_{\vec\mu}\right)=\triv_{W_n}\otimes \tilde\tau_{\vm\left(\triv_G\right)}\cong \tilde\tau_{\vm\left(\triv_G\right)},
    \]
    where, as above, $\tau_\mu$ is the $S_n$-irreducible representation corresponding to the partition $\mu$, which becomes the $W_n$-representation $\tilde\tau_\mu$ by lifting. 
\end{proof}

Furthermore, it turns out that the representations from Corollary \ref{cor:trivial-color=lift} are the only irreducible representations that have non-trivial $G^n$-invariant subspaces.
\begin{lemma}\label{lemma:triv isotypic of G^n comes from lifts}
    Let $\rho:W_n\to GL(V)$ be any finite-dimensional $W_n$-representation and consider its restriction to $G^n$. Denote by $V_\theta$ the isotypic component corresponding to $\theta\in\Irr(G^n)$, so $V=\bigoplus_{\theta\in\Irr (G^n)}V_\theta$. Then $V_{\triv_{G^n}}$ is a sub-representation of $\rho$.
\end{lemma}

\begin{proof}
    As $S_n$ permutes the coordinates of $G^n$, it carries $V_{\theta}$ to $V_{\sigma.\theta}$, where\linebreak
    $(\sigma.\theta)(a_1,\dots,a_n):=\theta(a_{\sigma^{-1}(1)},\dots,a_{\sigma^{-1}(n)})$ for $\sigma\in S_n$.
    In particular, the trivial representation of $G^n$, $\mathrm{triv}_{G^n}$ is fixed by all of $S_n$, so
    \[
        V_{\triv_{G^n}} = \left\{v\in V: \forall a\in G^n,~ a\cdot v=v\right\}
    \]
    is a $W_n$--invariant subspace on which $G^n$ acts trivially. 
\end{proof}

Note that using the notation of Lemmas \ref{lem:restriction-Gn-type} and \ref{lemma:triv isotypic of G^n comes from lifts}, the argument in the last proof shows that for every $\theta\in\Irr(G^n)$, the sum $\bigoplus_{\theta'\sim\theta}V_{\theta'}$ is a sub-representation of $\rho$.

From Corollary~\ref{cor:trivial-color=lift}, we know that $G^n$ acts trivially on an (entire) irreducible representation $\rho$ if and only if it is a lift of a representation of $S_n$, and along with Lemma~\ref{lemma:triv isotypic of G^n comes from lifts} this implies the following:
\begin{corollary}\label{cor:non lift irreps have no G^n trivial}
    Let $\rho:W_n\to GL(V)$ be an irreducible representation which is \textbf{not} a lift from $S_n$. Then its $G^n$-trivial part is $\{0\}$.
\end{corollary}

\subsection{Properties of $\JBG$ and $\LBG$}\label{subsec: average properties}

Let $B\subseteq[n]$. In addition to the subgroups $S_B,W_B\le W_n$ mentioned above, we also define the subgroup
$G^B\le G^n\le W_n$ consisting of the elements of $G^n$ with $e_G$ in the coordinates outside $B$ (this notation is justified as these are given by functions $B\to G$). 
It is straightforward to check that \(W_B= G^B\rtimes S_B\). 

Recall the notation $J_B=\frac{1}{|S_B|}\sum_{\sigma\in S_B}\sigma\in\CSn$ from \eqref{eq:def of J_B}, which we also think of as an element of $\CWn$ via the embedding $S_n\le W_n$. Recall also the notation $\JBG$ from \eqref{eq: def of wreath J_B}.
The following is standard, as $J_B$ and $\JBG$ are averaging operators over the subgroups $S_B$ and $W_B$ of $W_n$, respectively, and every irreducible representation of $W_n$ is unitary as a finite-dimensional representation.

\begin{claim}\label{claim: JB is an orthogonal projection}
Let $(\rho,V)\in\Rep{W_n}$ be a (unitary) representation. Then $\rho(J_B)$ is the orthogonal projection onto $V ^{S_B}=\{v\in V\,\mid\,\rho(\sigma)v=v~~\forall \sigma\in S_B\}$.
Likewise, $\rho(\JBG)$ is the orthogonal projection onto $V^{W_B}=\{v\in V\,\mid\,\rho(w)v=v~~\forall w\in W_B\}$.
\end{claim}
For operators $T,S:V\to V$ on an inner product space $\left(V,\langle\cdot\rangle\right)$, we use the notation $T\succeq S$ to indicate that $T-S$ is positive semidefinite, that is, $\langle (T-S)v,v\rangle\ge0$ for every $v\in V$. 

\begin{corollary}\label{cor: LBG is PSD}
    Let $(\rho,V)\in\Rep{W_n}$ be a (unitary) representation. The Laplacian operator \(\rho(\LBG)\) is Hermitian and positive semidefinite, that is, 
    \begin{equation}
    \rho\left(\LBG\right) \succeq 0.
    \end{equation}
\end{corollary}
\begin{proof}
Recall that $\LBG=\sum_B c_B(e-\JBG)$.
By the claim above, $\rho(\JBG)$ is an orthogonal projection, which is always positive semidefinite as its eigenvalues are in $\set{0,1}$. This makes $\rho(e-\JBG) = I - \rho(\JBG) \succeq 0$ the projection onto its orthogonal complement, which is again positive semidefinite. We are done as the weights $c_B$ are assumed to be non-negative.
\end{proof}

Since $W_B = G^B\rtimes S_B$, giving $|W_B|=|S_B|\cdot|G|^{|B|}$, we also have the second (and final) part of Claim~\ref{claim:triv} by summing over the $|G|^{|B|}$ elements in every fiber of the projection, using Corollary~\ref{cor:trivial-color=lift}.
\begin{corollary}\label{corollary:laplacians-equal-on-lifts}
    Let $\mu\in\Par_n$ and let $\tilde{\tau_\mu}=\tau_\mu\circ\pi$ be the corresponding lift of the $\mu$-indexed irreducible representation of $S_n$ to $W_n$. Then for every weighted hypergraph $\Gamma = ([n],c)$,
    \[
        \tilde{\tau_\mu}\left(\LBG\right)=\tau_\mu\left(\LL_\Gamma\right)
    \]
    as elements of $\End(V)$ for $V$ the shared vector space of $\tau_\mu$ and $\tilde{\tau_\mu}$.
\end{corollary}

\begin{proof}
It is enough to show that $\tilde{\tau_\mu}(\JBG)=\tau_{\mu}(J_B)$ for every $B\subseteq[n]$.
Fix such $B$. As mentioned above, $W_B=G^B\rtimes S_B$, hence $\pi(W_B)=S_B$ and $|W_B|=|G^B|\cdot |S_B|=|G|^{|B|}\cdot|S_B|$.
Therefore, each $\sigma\in S_B$ has $|\pi^{-1}(\sigma)\cap W_B|=|G|^{|B|}=\frac{|W_B|}{|S_B|}$.
Grouping the sum in $\JBG$ by $\sigma:=\pi(w)\in S_B$,
\begin{align*}
\tilde{\tau_\mu}\left(\JBG\right)
&=\frac{1}{|W_B|}\sum_{\sigma\in S_B}\sum_{\substack{w\in\pi^{-1}\left(\sigma\right)\cap W_B}}\tau_\mu\left(\sigma\right)
=\frac{1}{|S_B|}\sum_{\sigma\in S_B}\tau_\mu\left(\sigma\right)
=\tau_\mu\left(J_B\right).\qedhere
\end{align*}
\end{proof}

Finally, we show that the hyperedge averages $\JBG$ commute with $G^n$,
\begin{lemma}\label{lem:JBG-commutes-Gn}
For every $a\in G^n$ and every $B\subseteq[n]$,
\[
    \left(a;e_{S_n}\right)\,\JBG\,\left(a;e_{S_n}\right)^{-1}\ =\ \JBG.
\]
Consequently, for any (unitary) representation $\rho$, the operator $\rho(\JBG)$ commutes with each element in $\set{\rho\left(\textbf{g}\right):\textbf{g}\in G^n}$.
\end{lemma}

\begin{proof}
Conjugation by $(a;e_{S_n})$ preserves the subgroup $W_B\le W_n$ (it does not change the set of coordinates outside $B$), hence it is a permutation on $W_B$, and preserves its $\CWn$-average, $\JBG$.
\end{proof}

\section{Proofs of the Main Results}\label{sec:main}
Recall that Theorem~\ref{thm:main} states that for a hypergraph  $\Gamma = ([n],c)$ with non-negative weights, we have 
\[
\ls\left(\Gamma,\regSn\right)=\ls\left(\Gamma,\regWn\right).
\]
The proof follows from the two propositions below.
First we establish an upper bound on the smallest eigenvalue of the lift of the standard representation of $S_n$:
\begin{proposition}\label{prop:star}
    Let $\Gamma=([n],c)$ be a weighted hypergraph. Then
    \[
        \lm\left(\Gamma,\tau_{\left(n-1,1\right)}\right)=\lm\left(\Gamma,\tilde\tau_{\left(n-1,1\right)}\right) \le 
        \min_{i\in[n]}\left(\sum_{B\ni i} c_B\right),
    \]
    and unless there is an almost-isolated vertex in $\Gamma$ (see Page~\pageref{page:almost-isolated}), the inequality is strict.
\end{proposition}

\noindent Then we show that this upper bound is a lower bound on the eigenvalues of any irreducible representation outside $\Irr_0(W_n)$. By Corollary~\ref{cor:trivial-color=lift}, these are the irreducible representations which are not lifts of an irreducible representation of $S_n$.

\begin{proposition}\label{prop:gap from triv diagonal}
    Let $\Gamma = ([n],c)$ be a weighted hypergraph and let $\rho:W_n\to GL(V)$ be an irreducible representation of $W_n$ which is not a lift of an irreducible representation of $S_n$. Then 
    \[ 
        \min_{i\in[n]}\left(\sum_{B\ni i} c_B\right) \le 
        \lm\left(\Gamma,\rho\right).
    \]
\end{proposition}

\begin{proof}[Proof of Theorem \ref{thm:main} assuming Propositions~\ref{prop:star} and \ref{prop:gap from triv diagonal}]
    Let $\rho$ be any irreducible representation of $W_n$ which is not a lift of an irreducible representation of $S_n$. Then,
    \[
        \lm\left(\Gamma,\tau_{(n-1,1)}\right) \stackrel{\mathrm{Prop.~}\ref{prop:star}}{\le} 
        \min_{i\in[n]}\left(\sum_{B\ni i} c_B\right)\stackrel{\mathrm{Prop.~}\ref{prop:gap from triv diagonal}}{\le} \lm\left(\Gamma,\rho\right).
    \]
    Therefore, the smallest non-trivial eigenvalue of the Laplacian of $\Gamma$ in the regular representation is obtained among the irreducible representations in $\Irr_0(W_n)$:
    \[
         \ls\left(\Gamma,\regWn\right) = 
         \min_{\triv_{W_n}\ne\rho\in\Irr_0(W_n)}\lm\left(\Gamma,\rho\right) \stackrel{\mathrm{Cor.~}\ref{corollary:laplacians-equal-on-lifts}}{=}
         \min_{\triv_{S_n}\ne\tau\in\Irr\left(S_n\right)}\lm\left(\Gamma,\tau\right) =
         \ls\left(\Gamma,\regSn\right),
    \]    
    and by Proposition~\ref{prop:star}, unless there is an almost-isolated vertex, the inequality is strict - for every $\rho\in\Irr(W_n)\setminus\Irr_0(W_n)$,
    \[
    \lambda_{\min}\left(\Gamma,\rho\right)>\lambda_{\min}^*\left(\Gamma,\regWn\right).
    \]
\end{proof}

\subsection{An upper bound on the smallest eigenvalue of $\tilde\tau_{(n-1,1)}$}
In this subsection, we prove an upper bound on the smallest eigenvalue of the ($S_n$ standard) representation $\tilde\tau_{(n-1,1)}$ on any weighted hypergraph $\Gamma$.

\begin{proof}[Proof of Proposition~\ref{prop:star}]
Using Corollary~\ref{corollary:laplacians-equal-on-lifts}, we have the equality $\lm\left(\Gamma,\tau_{(n-1,1)}\right)=\lm\left(\Gamma,\tilde\tau_{(n-1,1)}\right)$, so it suffices to consider $\tau_{(n-1,1)}$.

We work in the usual model of the standard $S_n$--representation defined by $(n-1,1)$,
\[
    V\ :=\ \left\{x=\left(x_1,\ldots,x_n\right)\in\mathbb{C}^n:\ \sum_{i=1}^n x_i=0\right\},
\]
with $S_n$ acting by permuting coordinates. This realizes $\tau_{(n-1,1)}$ as a unitary representation with respect to the standard Hermitian inner product on $\mathbb{C}^n$ restricted to $V$.

Fix $i_0\in[n]$ and set
\[
    \vio\ :=\ e_{i_0}-\frac{1}{n}\sum_{j=1}^n e_j\ \in\ V,
\]
where $(e_1,\ldots,e_n)$ is the standard basis of $\mathbb{C}^n$.

For each $B\subseteq[n]$, the operator $I-\tau_{(n-1,1)}(J_B)\in\mathrm{End}(V)$ is the orthogonal projection onto $(V^{S_B})^\perp$ (see Claim~\ref{claim: JB is an orthogonal projection}), hence
\begin{equation}\label{eq:proj-bound}
    0\ \le\ \left\langle \left(I-\tau_{\left(n-1,1\right)}\left(J_B\right)\right)v, v\right\rangle\ \le\ \langle v,v\rangle
    \qquad\forall v\in V,
\end{equation}
where equality on the right occurs if and only if $v\in (V^{S_B})^{\perp}$.
If $i_0\notin B$, then $\vio$ is constant on the coordinates in $B$, so it is fixed by $S_B$, namely, $\vio\in V^{S_B}$. Hence,
\begin{equation}\label{eq:zero-io}
    \left(I-\tau_{\left(n-1,1\right)}\left(J_B\right)\right)\vio=0
    \qquad\text{whenever } i_0\notin B.
\end{equation}
Note that $V^{S_B}$ is the subspace where the $B$-coordinates are all equal, that is,
\[
V^{S_B}=\set{v\in V:\exists c~~\mathrm{s.t.}~\forall k\in B,~\langle v, e_k\rangle=c}.
\]
Consequently, 
\[
\left(V^{S_B}\right)^{\perp}=\set{v\in V:\sum_{k\in B}\langle v,e_k\rangle = 0 ~~ \mathrm{and}~~\forall l\notin B,~\langle v,e_l\rangle = 0}.
\]
Thus, $\vio\notin (V^{S_B})^{\perp}$ whenever $i_0\in B\neq[n]$, and therefore in those cases,
\begin{equation}\label{eq:vio strict ineq}
\langle \left(I-\tau_{\left(n-1,1\right)}\left(J_B\right)\right)\vio, \vio \rangle < \langle\vio,\vio\rangle
\end{equation} 
for $i_0\in B$. So
\begin{align*}
\left\langle \tau_{\left(n-1,1\right)}\left(\LL_\Gamma\right)\vio,\ \vio\right\rangle
&=\sum_{B\subseteq[n]}c_B\left\langle (I-\tau_{(n-1,1)}(J_B))\vio,\ \vio\right\rangle \\
&=\sum_{B\ni i_0}c_B\left\langle (I-\tau_{(n-1,1)}(J_B))\vio,\ \vio\right\rangle \\
&\le \sum_{B\ni i_0}c_B\,\langle \vio,\vio\rangle,
\end{align*}
where the second equality is due to~\eqref{eq:zero-io} and the last inequality uses \eqref{eq:proj-bound}. The last inequality is strict whenever there is a non-zero weight for some $B\subsetneq[n]$ containing $i_0$ by \eqref{eq:vio strict ineq}. 
Dividing by $\langle \vio,\vio\rangle>0$ gives a Rayleigh quotient bound
\[
    \frac{\left\langle \tau_{(n-1,1)}(\LL_\Gamma)\vio,\ \vio\right\rangle}{\langle \vio,\vio\rangle}
    \ \le\ \sum_{B\ni i_0}c_B.
\]
Since $\tau_{(n-1,1)}(\LL_\Gamma)$ is Hermitian and positive semidefinite, its minimal eigenvalue is the minimum of the Rayleigh quotients over non-zero vectors in $V$, hence
\[
    \lm\!\left(\Gamma,\tau_{(n-1,1)}\right)\ \le\ \sum_{B\ni i_0}c_B.
\]
Taking $i_0$ that attains the minimum value, and assuming that it is not almost-isolated, we get a strict inequality
\[
    \lm\!\left(\Gamma,\tau_{(n-1,1)}\right)\ <\ \sum_{B\ni i_0}c_B=\min_i\sum_{B\ni i}c_B.\qedhere
\]
\end{proof}

\subsection{A lower bound on the smallest eigenvalue of  $\rho\in\Irr (W_n)\setminus \Irro(W_n)$}

We now prove Proposition~\ref{prop:gap from triv diagonal}. 
Let $\rho:W_n\to GL(V)$ be a $W_n$-representation. 
As in Lemma~\ref{lemma:triv isotypic of G^n comes from lifts}, consider its restriction to $G^n$, and let the isotypic decomposition of $\Res_{G^n}^{W_n}\rho$ to subspaces be
\[
V=\bigoplus_{\theta\in\Irr\left(G^n\right)} V_\theta.
\]
Since $\rho$ is finite-dimensional (hence unitary), this is an orthogonal direct sum.

For $\theta=\theta_1\boxtimes\cdots\boxtimes\theta_n\in\Irr(G^n)$ define
\[
\supp\left(\theta\right)\ :=\ \{i\in[n]:\theta_i\not\simeq \triv_G\}.
\]

We view $\CSn$ as a subalgebra of $\CWn$ via the embedding
$S_n\hookrightarrow W_n$, $\sigma\mapsto (e_{G^n};\sigma)$.
In particular, for $B\subseteq[n]$, the element $J_B\in\CSn$ from \eqref{eq:def of J_B}
is identified with the same average in $\CWn$:
\[
J_B\ =\ \frac{1}{|S_B|}\sum_{\sigma\in S_B}\left(e_{G^n};\sigma\right)\ \in\ \CWn.
\]

We start by showing that $\supp(\theta)$ determines how $\JBG$ acts on $V_\theta$.
\begin{lemma}\label{lem:J-on-theta}
For every $\theta\in\Irr(G^n)$ and $B\subseteq[n]$,
\[
\rho\left(\JBG\right)\Big|_{V_\theta}\ =\
\begin{cases}
\rho\left(J_B\right)\big|_{V_\theta} & \text{if } B\cap \supp\left(\theta\right)=\emptyset,\\[2pt]
0 & \text{if } B\cap \supp\left(\theta\right)\neq\emptyset.
\end{cases}
\]
\end{lemma}

\begin{proof}
Fix $B\subseteq[n]$ and recall that $W_B=G^B\rtimes S_B$, so $|W_B|=|G|^{|B|}\,|S_B|$. Using
\eqref{eq:wreath-fact}, for $g\in G^B$ and $\sigma\in S_B$,
\[
\left(g;\sigma\right)=\left(e_{G^n};\sigma\right)\,\left(g\circ\sigma;e_{S_n}\right),
\]
hence
\[
\JBG
=\frac{1}{|W_B|}\sum_{\sigma\in S_B}\sum_{g\in G^B}\left(e_{G^n};\sigma\right)\,\left(g\circ\sigma;e_{S_n}\right)
=\frac{1}{|S_B|}\sum_{\sigma\in S_B}\left(e_{G^n};\sigma)\cdot
\left(\frac{1}{|G|^{|B|}}\sum_{g\in G^B}(g;e_{S_n}\right)\right),
\]
where we used the fact that $g\mapsto g\circ\sigma$ is a bijection of $G^B$ for every fixed $\sigma\in S_B$.

On the $G^n$--isotypic component $V_\theta$, the restriction of $\rho$ to $G^n$ is a direct sum of copies of
$\theta=\theta_1\boxtimes\cdots\boxtimes\theta_n$, so the operator
$\frac{1}{|G|^{|B|}}\sum_{g\in G^B}\rho\left(g;e_{S_n}\right)$ acts as $\frac{1}{|G|}\sum_{h\in G}\theta_i(h)$ on coordinates $i\in B$ and as the identity on coordinates outside of $B$.
By Schur's lemma,
$\frac{1}{|G|}\sum_{h\in G}\theta_i(h)$ is the identity if $\theta_i\simeq\triv_G$, and zero otherwise.
Therefore, if $B\cap \supp(\theta)=\emptyset$, the operator corresponding to this $G^B$--average is the identity on $V_\theta$ and

\[
\rho\left(\JBG\right)\Big|_{V_\theta}
=\frac{1}{|S_B|}\sum_{\sigma\in S_B}\left(\rho\left(e_{G^n};\sigma\right) \cdot I\right)\big|_{V_\theta}
=\rho(J_B)\big|_{V_\theta}.
\]

On the other hand,  if $B\cap \supp(\theta)\neq\emptyset$, there is a tensor coordinate that maps to $0$ and the $G^B$--average vanishes. Hence $\rho(\JBG)\big|_{V_\theta}=0$.
\end{proof}

\begin{corollary}\label{cor:L-preserves-blocks}
Let $\rho:W_n\to GL(V)$ be an irreducible $W_n$-representation and write the $G^n$--isotypic decomposition $V=\bigoplus_{\theta\in\Irr(G^n)} V_\theta$.
Then for any hypergraph $\Gamma$, each $V_\theta$ is $\rho(\LBG)$--invariant and
\[
\operatorname{spec}\!\left(\rho\left(\LBG\right)\right)\ =\ \bigsqcup_{\theta\in\Irr\left(G^n\right)}
\operatorname{spec}\!\left(\rho\left(\LBG\right)\Big|_{V_\theta}\right)
\]
(with multiplicity), where $\operatorname{spec}\left(T\right)$ denotes the complex spectrum of an operator $T$.
\end{corollary}

\begin{proof}
Assume that $\vm$ is the multi-partition such that $\rho=\rho_{\vm}$.
It suffices to show that for every $B\subseteq[n]$ and every $\theta\in\Irr(G^n)$, $\rho(\JBG)$ preserves $V_\theta$. Indeed, by \eqref{eq: def of wreath laplacian}, this would imply that $\rho(\LBG)$ preserves $V_\theta$ as well, and since the decomposition into $G^n$--isotypic components is orthogonal, the spectrum of $\rho(\LBG)$ is the disjoint union of the spectra of its restrictions to the blocks $V_\theta$.

Fix $B\subseteq[n]$ and $\theta\in\Irr(G^n)$. 
Taking $P_\theta = \frac{\dim\theta}{|G^n|}\sum_{g\in G^n}\chi_\theta(g^{-1})g\in \CWn$ -- the element projecting onto the $\theta$-isotypic part of every $G^n$-representation -- we have that $\rho\left(P_\theta\right) V = V_\theta$ and by Lemma~\ref{lem:JBG-commutes-Gn}, $ \JBG P_\theta  = P_\theta \JBG$. Hence 
$$\rho\left(\JBG\right) V_\theta = \rho\left(\JBG\right) \rho\left(P_\theta\right) V= \rho\left(P_\theta\right) \rho\left(\JBG\right) V \subseteq V_\theta.$$ 
\end{proof}

Finally, we prove the main proposition.

\begin{proof}[Proof of Proposition~\ref{prop:gap from triv diagonal}]
    Let $\rho\in\Irr(W_n) \setminus \Irro(W_n)$, let $\vm$ be the multi-partition corresponding to $\rho$, and let $V= \bigoplus_{\theta \in \Irr(G^n)} V_\theta$ be the decomposition as above. By Corollary~\ref{cor:non lift irreps have no G^n trivial}, as $\rho$ is not a lift of an irreducible representation of $S_n$, its $G^n$-trivial part is ${0}$. Namely, for every $\theta\in\Irr(G^n)$ with $V_\theta\ne0$,  $\supp(\theta)\neq\emptyset$. By Corollary~\ref{cor:L-preserves-blocks}, it suffices to show that 
    \[
        \lm\left(\rho\left(\LBG\right)\Big|_{V_\theta}\right) \geq \min_{i\in[n]}\sum_{B\ni i}c_B
    \]
    for every $\theta\in\Irr(G^n)$ with $\supp(\theta)\neq\emptyset$.
    Indeed, 
    \begin{eqnarray*}
        \rho\left(\LBG\right)\Big|_{V_\theta} &=& \sum_{B\subseteq [n]} c_B\left(I-\rho\left(\JBG\right)\Big|_{V_\theta}\right) \\ 
        &\stackrel{Lemma~\ref{lem:J-on-theta}}{=}& \sum_{B\cap\supp\left(\theta\right) =\emptyset}c_B\cdot\left(I-\rho\left(J_B\right)\mid_{V_\theta}\right) + \sum_{B\cap\supp\left(\theta\right)\ne\emptyset} c_B\cdot I \\
        &\stackrel{Claim~\ref{claim: JB is an orthogonal projection}}{\succeq}& \sum_{B\cap\supp\left(\theta\right)\ne\emptyset} c_B\cdot I.
    \end{eqnarray*}    
    We are done as for every $i_0\in\supp(\theta)\neq\emptyset$, 
    \[
        \sum_{B\cap\supp\left(\theta\right)\ne\emptyset} c_B \ge 
        \sum_{B\ni i_0} c_B \ge 
        \min_{i\in[n]}\sum_{B\ni i} c_B.\qedhere
    \]
\end{proof}

\bibliographystyle{alpha}
\bibliography{bib}

@article{AldousCaputoDurrettHolroydJungPuha2021Liggett,
  author  = {Aldous, D. and Caputo, P. and Durrett, R. and Holroyd, A. E. and Jung, P. and Puha, A. L.},
  title   = {The Life and Mathematical Legacy of {T}homas {M}. {L}iggett},
  journal = {Notices Amer. Math. Soc.},
  volume  = {68},
  number  = {1},
  pages   = {67--79},
  year    = {2021},
  doi     = {10.1090/noti2203},
  eprint  = {2008.03137},
  archivePrefix = {arXiv},
  primaryClass  = {math.PR},
}

@unpublished{AlonGhosh25,
  author        = {G. Alon and S. Ghosh},
  title         = {Spectral gap for the signed interchange process with arbitrary sets},
  note       = {preprint arXiv:2510.04244},
  year          = {2025}
}

@article{ALON_KOZMA_PUDER_2025, 
    title={On the {A}ldous-{C}aputo Spectral Gap Conjecture for Hypergraphs}, 
    volume={179}, 
    DOI={10.1017/S0305004125000179}, 
    number={2}, 
    journal={Math. Proc. Cambridge Philos. Soc.}, 
    author={Alon, G. and Kozma, G. and Puder, D.}, 
    year={2025}, 
    pages={259–298}
}

@unpublished{AP26,
  author  = {G. Alon and D. Puder},
  title   = {Aldous-type Spectral Gaps in Unitary Groups},
  note    = {Preprint, arXiv:2603.00353},
  year    = {2026}
}

@inproceedings{bristiel2024entropy,
  title={Entropy inequalities for random walks and permutations},
  author={Bristiel, A. and Caputo, P.},
  booktitle={Ann. Inst. Henri Poincar{\'e} Probab. Stat.},
  volume={60},
  pages={54--81},
  year={2024}
}

@article{CLR10,
  author  = {Caputo, P. and Liggett, T. M. and Richthammer, T.},
  title   = {Proof of {A}ldous' spectral gap conjecture},
  journal = {J. Amer. Math. Soc.},
  year    = {2010},
  volume  = {23},
  number  = {3},
  pages   = {831--851},
  doi     = {10.1090/S0894-0347-10-00659-4}
}

@article{Cesi2016Octopus,
  author  = {Cesi, F.},
  title   = {A Few Remarks on the {O}ctopus Inequality and {A}ldous' Spectral Gap Conjecture},
  journal = {Comm. Algebra},
  volume  = {44},
  number  = {1},
  pages   = {279--302},
  year    = {2016},
  doi     = {10.1080/00927872.2014.975349},
  eprint  = {1310.6156},
  archivePrefix = {arXiv},
  primaryClass  = {math.PR},
}

@article{CesiWeylB,
  author    = {F. Cesi},
  title     = {On the spectral gap of some {C}ayley graphs on the {W}eyl group ${W(B_n)}$},
  journal   = {Linear Algebra Appl.},
  volume    = {586},
  pages     = {274--295},
  year      = {2020},
  doi       = {10.1016/j.laa.2019.10.038},
  eprint    = {1807.11833},
  archivePrefix = {arXiv},
  primaryClass  = {math.CO}
}

@book{ceccherini2014representation,
  title={Representation theory and harmonic analysis of wreath products of finite groups},
  author={Ceccherini-Silberstein, T. and Scarabotti, F. and Tolli, F.},
  volume={410},
  year={2014},
  publisher={Cambridge University Press}
}

@article{Ghosh23,
  author  = {Ghosh, S.},
  title   = {Aldous-type spectral gap results for the complete monomial group},
  journal = {Ann. Inst. Henri Poincaré Probab. Stat.},
  year    = {2023},
  note  = {to appear, also at arXiv:2309.12154}
}

@article{PP18,
  author  = {Parzanchevski, O. and Puder, D.},
  title   = {Aldous's spectral gap conjecture for normal sets},
  journal = {Trans. Amer. Math. Soc.},
  year    = {2020},
  volume  = {373},
  number  = {10},
  pages   = {7067--7086},
  doi     = {10.1090/tran/8155}
}

@mastersthesis{Piras2010,
  author  = {Piras, D.},
  title   = {Generalizations of {A}ldous' Spectral Gap Conjecture},
  school  = {Universit\`a degli Studi Roma Tre},
  year    = {2010},
  type    = {Tesi di Laurea},
  note    = {Advisor: P. Caputo},
  url     = {https://www.mat.uniroma3.it/scuola_orientamento/alumni/laureati/piras/sintesi.pdf},
}

@book{Zelevinsky81,
  author    = {Zelevinsky, A. V.},
  title     = {Representations of Finite Classical Groups: A Hopf Algebra Approach},
  series    = {Lecture Notes in Mathematics},
  volume    = {869},
  publisher = {Springer},
  year      = {1981}
}

~\\
Niv Levhari\\
School of Mathematical Sciences, Tel Aviv University, Tel Aviv, 6997801, Israel\\
\noindent\texttt{nlevhari@gmail.com}\\
\\
Doron Puder\\
School of Mathematical Sciences, Tel Aviv University, Tel Aviv, 6997801, Israel\\
and IAS Princeton, School of Mathematics, 1 Einstein Drive, Princeton, NJ 08540, USA\\
\texttt{doronpuder@gmail.com}~\\

\end{document}